\documentclass[oneside,a4paper]{amsart}
	\usepackage{carlocommands}
\usepackage{lmodern}
	\usepackage[T1]{fontenc}

\begin{document}
	\title[Cut and Initial Sequents]{Cut elimination for systems of \\ transparent truth with restricted initial sequents}
	\author{Carlo Nicolai\\King's College London}
	\date{}
	\thanks{Thanks to Peter Schr\"oder-Heister for pointing me to his work on definitional reflection that studies the interaction between cut, contraction, and restricted initial sequents studied in this work. Thanks to Andreas Fjellstad for clarifying the role of identity axioms in the systems I discuss. I thank Volker Halbach, Graham Leigh, Beau Mount, Luca Tranchini for discussing with me the ideas contained in the paper. Special thanks to Luca Castaldo, Martin Fischer, Lorenzo Rossi for detailed comments.  }
	\address{\normalfont carlonicolai6@gmail.com\;;\;carlo.nicolai@kcl.ac.uk}

\maketitle

\begin{abstract}
	The paper studies a cluster of systems for fully disquotational truth based on the restriction of initial sequents. Unlike well-known alternative approaches, such systems display both a simple and intuitive model theory and remarkable proof-theoretic properties. We start by showing that, due to a strong form of invertibility of the truth rules, cut is eliminable in the systems via a standard strategy supplemented by a suitable measure of the number of applications of truth rules to formulas in derivations. Next, we notice that cut remains eliminable when suitable arithmetical axioms are added to the system. Finally, we establish a direct link between cut-free derivability in infinitary formulations of the systems considered and fixed-point semantics. Noticeably, unlike what happens with other background logics, such links are established without imposing any restriction to the premisses of the truth rules.
	
\end{abstract}


\section{Introduction}

Due to the Liar paradox, fully disquotational approaches to truth -- that is, satisfying the rules ($\T${\sc l}) and ($\T${\sc r}) below -- require a non-classical logical treatment. Among the nonclassical options, a standard approach is to restrict operational rules for connectives that play a crucial role in the derivation of the inconsistency, such as negation or material implication. To this family of approaches belong the various paracomplete or paraconsistent accounts of fully disquotational truth defended in the literature (see e.g.~\cite{fie08,pri05,kre88,bea09,haho06}). 

Formal systems for transparent truth based on restrictions of operational rules and  featuring unrestricted rules for semantic notions do not sit well with standard strategies to fully or partially eliminate applications of the cut rule.\footnote{Actually, this generalizes to partial approaches to truth based on supervaluations. We shall elaborate on this point later.} To explain why this is so, let us focus on the case of unrestricted truth rules
\begin{align*}
	& \AxiomC{$\Gamma, \vphi\Ra \Delta$}\RLB{($\T${\sc l})}
			\uinf{\Gamma, \T\corn{\vphi}\Ra \Delta}
		\DisplayProof
	&& \AxiomC{$\Gamma\Ra \vphi,\Delta$}\RLB{($\T${\sc r})}
			\uinf{\Gamma\Ra \T\corn{\vphi}, \Delta}
		\DisplayProof
\end{align*}
In the rules, $\corn{\cdot}$ is a quotation device that yields a canonical name for each sentence of the language. 
	When one wants to eliminate a cut on truth ascriptions   $\T\corn{\vphi}$ that are both obtained from {\sc ($\T$-l)} and {\sc ($\T$-r)}, a natural thought is to cut on the sentence $\vphi$ in the premises of the application of these rules. However it's clear that $\T\corn{\vphi}$ is an atomic formula, whereas $\vphi$ may be an extremely (logically) complex sentence. Therefore a simple induction on the logical complexity of the cut formula, on which cut-elimination strategies are traditionally based, will not work. One has also to keep track of the number of applications of truth rules and induct over it in the main induction hypothesis. 

There are several alternatives for such tracking devices. One option is to assign a measure to sequents, i.e.~to nodes in the derivation tree. However, in the context of logics which restrict the operational rules, this strategy can only be carried out if one {\it restricts} the truth rules, by disallowing contexts in the premises. Once this restriction takes place, cut can be fully eliminated. Of course, the price to pay is the adoption of properly weaker truth rules. This is for instance the strategy considered in \cite{can90,figr18} in the context of Strong Kleene logic and supervaluational logic respectively. Alternatively, one can restrict the contraction rule, keep the node-based measure of applications of semantic rules, and still obtain a full cut elimination proof \cite{gri82,can03}.\footnote{\cite{zar11} has also presented a cut-elimination argument for an (infinitary) transparent theory of truth over a contraction-free logic. However, both \cite{daro18} and \cite{fje20} cast some doubts on the logical coherence and applicability of the proposal.} However, the restriction of contraction has its own drawbacks. For instance, whereas the systems based on the restriction of operational rules are sound -- and, in a suitably controlled environment, also complete -- with respect to a class of  fixed-point models  \cite{kri75},\footnote{More on fixed point semantics in  \S\ref{sec:seminf}.} no such link exists between contraction-free systems and fixed-point semantics or natural alternatives.

In this paper I consider a cluster of theories of transparent truth that display both a direct link with fixed-point semantics, but also desirable proof-theoretic properties culminating in the eliminability of cut. Such systems are based on a restriction of initial sequents to formulas not containing the truth predicate. The proof-theoretic arguments given below crucially rest on the adoption of a measure for formulas in derivations, called $\T$-complexity, that keeps track of the number of truth rules applied to ancestors of a single formula in the given proof. While the notion of $\T$-complexity is not new,\footnote{Similar measures of complexity have been considered by \cite{hal99} and \cite{lei15}.} it is its combination with the restriction of initial sequents in the context of transparent truth that is the main focus of the paper. Such connection has been studied already for a propositional logic extended with rules for definitional reflection in \cite{sch16}.\footnote{The key ideas of this paper were presented in T\"ubingen in 2017, where Peter Schr\"oder-Heister pointed to his independent work on the idea. The current shape of the paper and results benefited greatly from the study of Schr\"oder-Heister's work on definitional reflection. Such exchanges also are at the root of the formulation of an infinitary Tait system with restricted initial sequents, akin to $\lpcinf$ below, in Martin Fischer's Habilitation Thesis, \emph{Modal Predicates and Their Interaction}, Munich, 2018.}

	\subsection*{Plan and structure of the paper}
	
	In section \ref{sec:lpc}, I study the proof-theory of a `logic' of truth $\lpc$, that is a system with no non-logical initial sequents and rules besides {\sc ($\T$-l)} and {\sc ($\T$-r)}. The section focuses on the definition of the main measure for application of truth rules called $\T$-complexity (Definition \ref{dfn:tcompl}), the proof of the strong invertibility property of the main $\lpc$-rules (Lemma \ref{lem:invlog}), and culminates with the eliminability of cut in $\lpc$ essentially achieved in Lemma \ref{lem:redlpt}. In  the short section \ref{sec:ariext}, we extend the results of section \ref{sec:lpc} to extensions of $\lpc$ with (geometric) arithmetical axioms by employing the study of the proof-theory of geometric axioms from \cite{nevo11}. This yields a uniform conservativeness proof of local truth rules over the base theory (Proposition \ref{prop:conse}). Section \ref{sec:seminf} studies the connection of cut-free provability and an infinitary extension of $\lpc$ with fixed point semantics (Lemmata \ref{lem:grocon} and \ref{lem:mfpsys}). To achieve this, $\T$-complexity is extended to the transfinite, and cut elimination is proved for an infinitary extension of $\lptn$ (Proposition \ref{prop:ceinf}). 
	
	As the reader will notice, the cut elimination strategy introduced in \S\ref{sec:lpc} features prominently also in the subsequent sections. Of course, an alternative way of presenting the paper would have been to start with the arithmetical or the infinitary setting, and then inferring the results of \S\ref{sec:lpc} as immediate corollaries. The current structure of the paper is motivated by the intention of presenting  the main structural lemmata in a simple setting in \S\ref{sec:lpc}, so that in the subsequent sections the focus could be mainly on the adjustments required by richer frameworks and on other properties such as the connections with fixed-point semantics. 


\section{Logics for transparent truth with restricted initial sequents}\label{sec:lpc}

We start with a first-order language $\mc{L}$ with logical constants $\neg,\land,\forall,\bot,\top$. We let:
%
%
	\begin{align*}
		&\lt:=\mc{L}\cup\{\T\},\text{for $\T$ a unary predicate symbol.}
	\end{align*}
We write ${\rm AtFml}_{\lt}$ for the set of atomic formulas of $\lt$, and ${\rm Sent}_{\lt}$ for the set of sentences of $\lt$. It is useful \emph{not} to regard $\top,\bot$ as members of ${\rm AtFml}_{\lt}$. 
%
%
The \emph{logical complexity} $\lth{\vphi}$ of a formula $\vphi$ of $\lt$  is defined inductively  as the number of nodes in the maximal branch of its syntactic tree:
	\[
		\lth{\vphi}=\begin{cases}
						0,&\text{if $\vphi$ is atomic or $\bot,\top$},\\
						\lth{\psi}+1,&\text{if $\vphi\equiv\neg\psi$ or $\vphi\equiv \forall x\psi$,}\\
						\maxf(\lth{\psi},\lth{\chi})+1,&\text{if $\vphi\equiv \psi\land\chi$.}
					\end{cases}
	\]

\noindent To properly formulate our truth rules in the simple setting studied in this section, we follow the standard practice of assuming that for any sentence $\vphi\in \lt$, there is a term $\corn{\vphi}$ playing the logical role of its name \cite{kre88,can03,rip12}. In general, there are good reasons to require much more than a simple, essentially metatheoretic quotation device and work with a fully fledged formal syntax in the background. We will see later on that much of our discussion can be transferred to such richer settings. 

In what follows, $\Gamma,\Delta,\Theta,\Lambda\ldots$ stand for \emph{finite multisets} of formulas of $\lt$ -- and the same notation will be employed for the different languages considered below. Expressions of the form $\Gamma \Ra \Delta$ are \emph{sequents}. We assume a standard notion of substitution and write $\Gamma(t/x)$ for the result of replacing all free occurrences of $x$ in all formulas in $\Gamma$ with the term $t$, which is assumed to be free for $x$ in such formulas. For a formula $\vphi$, we denote with ${\rm FV}(\vphi)$ the set of its free variables. ${\rm FV}(\Gamma)$ denotes the set of free variables in formulas in $\Gamma$. 

The system $\lpc$ is essentially characterized by all operational rules of classical logic, fully disquotational truth rules, and crucially by a restriction of initial sequents to principal formulas that are atomic and do not contain $\T$. In the terminology of \cite{trsc03}, $\lpc$ is a ${\bf G3}$ system. The label $\lpc$ stands for `logic of grounded truth'. This choice is informally motivated by the fact that one can read the sequent $\Gamma\Ra \Delta$ in $\lpc$ as stating that either some member of $\Gamma$ is determinately false, or some member of $\Delta$ is determinately true. This informal picture will be refined by the semantic considerations of Section \ref{sec:seminf} -- and Lemma \ref{lem:grocon} in particular. 
\begin{dfn}[$\lpc$]
The system $\lpc$ in $\lt$ features the following initial sequents and rules:\footnote{\label{foo:disexi}We omitted the standard ${\bf G3}$-rules for $\vee,\exists$, which are nonetheless admissible in the systems below by employing the usual definitions of $\vee$ and $\exists$ in terms of $\land,\neg,\forall$.}

\begin{align*}
	& \text{{\sc \small{(ref$^-$)}}}\;\;\;\AxiomC{$\Gamma, \vphi \Ra \vphi,\Delta$}\noLine
		\uinf{\textup{with $\vphi\in {\rm AtFml}_{\mc{L}}$} } \DisplayProof
		&&\AxiomC{$\Gamma\Ra \Delta,\vphi$}\AxiomC{$\vphi,\Gamma\Ra \Delta$}\Llb{cut}\BinaryInfC{$\Gamma \Ra \Delta$}\DisplayProof\\[10pt]
	&  \text{{\sc \small{($\top$)}}}\;\;\ax{\Gamma\Ra\top,\Delta}
	\DisplayProof
	&& \text{{\sc \small{($\bot$)}}}\ax{\Gamma,\bot\Ra\Delta}
	\DisplayProof\\[10pt]
	& \AxiomC{$\Gamma, \vphi\Ra \Delta$}\Llb{$\T${\sc l}}
			\uinf{\Gamma, \T\corn{\vphi}\Ra \Delta}
		\DisplayProof
	&& \AxiomC{$\Gamma\Ra \vphi,\Delta$}\Llb{$\T${\sc r}}
			\uinf{\Gamma\Ra \T\corn{\vphi}, \Delta}
		\DisplayProof\\[10pt]
	& \ax{\Gamma\Ra\vphi, \Delta}\Llb{$\neg$l}
		\uinf{\Gamma,\neg\vphi\Ra \Delta}
		\DisplayProof
		&&\ax{\Gamma,\vphi\Ra \Delta}\Llb{$\neg$r}
		\uinf{\Gamma\Ra\neg\vphi, \Delta}
		\DisplayProof
		\\[10pt]
	&\AxiomC{$\Gamma,\vphi,\psi\Ra \Delta$}\Llb{$\land$l}\UnaryInfC{$\Gamma,\vphi\land \psi\Ra \Delta$}\DisplayProof
		&&\AxiomC{$\Gamma\Ra \vphi,\Delta$}\AxiomC{$\Gamma\Ra\psi,\Delta$}\Llb{$\land$r}\BinaryInfC{$\Gamma \Ra \Delta,\vphi\land \psi$}\DisplayProof\\[10pt]
	&\AxiomC{$\Gamma,\forall x\vphi,\vphi(s/x)\Ra \Delta$}\Llb{$\forall$l}\UnaryInfC{$\Gamma,\forall x\vphi\Ra \Delta$}\DisplayProof
		&&\AxiomC{$\Gamma\Ra \vphi(y/x),\Delta$}\Llb{$\forall$r}\RLB{$y\notin{\rm FV}(\Gamma,\Delta,\forall x\vphi)$}\UnaryInfC{$\Gamma \Ra \Delta,\forall x\vphi$}\DisplayProof
\end{align*}
\end{dfn}

The following measures of complexity are also standard. We employ the usual notions of \emph{premisses} and \emph{conclusion} of rules, \emph{principal}, \emph{active}, \emph{side} formulas \cite{sch77,trsc03}:
	\begin{enumeratei}
		\item Given rules that are at most $\alpha$-branching,  the \emph{length} $d$ of a derivation $\mc{D}$ is
		\[
			{\rm sup}\{d_\gamma+1\sth \gamma<\beta\}
		\]
		where $\mc{D}_\gamma$ ($\gamma<\beta\leq \alpha$) are $\mc{D}$'s direct subderivations.
		\item The \emph{rank} of an application of {\sc cut} on $\vphi$ is $\lth{\vphi}+1$. The \emph{cut rank} of a derivation $\mc{D}$ is the maximum of the ranks of cut formulas in $\mc{D}$. 
	\end{enumeratei}

	It will sometimes be useful to refer directly to different occurrences of the same (qua syntactic object) formula in a derivation (cf.~\cite{sch16}). When writing, say, 
\begin{equation}\label{eq:occnot}
	\ax{\gamstrjn\Ra\delstrkm, \vphi}
	\uinf{\gamstrjon\Ra\delstrkom,\psi}
	\DisplayProof
\end{equation}
we assume that occurrences of $\gamma^j_i$, with $1\leq i\leq n$ correspond precisely to occurrences of $\gamma^{j+1}_i$ -- i.e.~they are distinct occurrences of the same formula -- and similarly for the $\delta$'s. As an abbreviation, this will be generalized to multisets of sentences: I occasionally write $\Gamma^j$ instead of $\gamstrjn$. It should be clear that superscripts are not part of the language.

The idea behind the following measure on proofs, that we call $\T$-complexity, plays an important role in recent proof-theoretic studies of primitive truth predicates \cite{hal99,lei15}. It essentially tracks the number of truth rules applied to formulas in derivations. If contraction is present, such measure is not easy to define and employ.\footnote{It is in fact the presence of contraction that led to error in its applications in \cite{hal99}, which are rectified by \cite{lei15}, but only for typed truth, not type-free truth.} We will see that the restriction of initial sequents and the absence of explicit contraction enable us to apply the notion of $\T$-complexity in the general case of type-free, disquotational truth. 
\begin{dfn}[$\T$-complexity]\label{dfn:tcompl}
	The ordinal $\T$-complexity of an occurrence of a formula $\vphi$ of $\lt$ in a derivation $\mc{D}$ in $\lpc$ -- in symbols, $\tau_{\mc{D}}(\vphi)$ is defined inductively as follows:
	\begin{enumeratei}
		\item $\tau_\mc{D}(\vphi)=0$ if $\vphi\in \mc{L}$;
		\item\label{case:weatc} If $\mc{D}$ contains only an initial sequent node ({\sc ref},$\top$,$\bot$), then $\tau_{\mc{D}}(\vphi)=0$ for all formulas in it. 
		\item If $\mc{D}$ ends with 
			\[
				\ax{\Gamma\Ra\Delta,\psi}
					\uinf{\Gamma \Ra \Delta,\T\corn{\psi}}
				\DisplayProof
			\]	
			then $\tau_\mc{D}(\T\corn{\psi})=\tau_\mc{D}(\psi)+1$ and the $\T$-complexity of the formulas in $\Gamma,\Delta$ is unchanged. Similarly for \textup{({\sc $\T$l})}.
		\item If $\mc{D}$ ends with 
		\[
			\ax{\Gamma\Ra\Delta,\vphi}
				\uinf{\neg\vphi,\Gamma\Ra\Delta}
			\DisplayProof
		\] 
		then $\tau_\mc{D}(\vphi)=\tau_\mc{D}(\neg\vphi)$ and the $\T$-complexity of the formulas in $\Gamma,\Delta$ is unchanged. Similarly for {\sc $(\neg$r)} and {\sc ($\forall$r)}.
		\item If $\mc{D}$ ends with
		\[
			\ax{\Gamma,\vphi,\psi\Ra\Delta}
				\uinf{\vphi\land\psi,\Gamma\Ra\Delta}
			\DisplayProof
		\] 
		then $\tau_\mc{D}(\vphi\land\psi)=\maxf(\tau_\mc{D}(\vphi),\tau_\mc{D}(\psi))$ and the
$\T$-complexity of the formulas in $\Gamma,\Delta$ is unchanged. 
		\item\label{case:tccr} If $\mc{D}$ ends with -- cf.~notational convention after \eqref{eq:occnot},
			\[
				\ax{\Gamma^j\Ra \vphi,\Delta^k}
				\ax{\Gamma^l\Ra \psi,\Delta^p}
					\binf{\Gamma\Ra \vphi\land\psi,\Delta}
				\DisplayProof 
			\]
			then $\tau_{\mc{D}}(\vphi\land\psi)=\maxf(\tau_\mc{D}(\vphi),\tau_\mc{D}(\psi))$, and 			\begin{align*}
				&\tau_{\mc{D}}(\gamma_i)=\maxf(\tau_{\mc{D}}(\gamma_i^{j_i}),\tau_{\mc{D}}(\gamma_i^{l_i})),&&1\leq i\leq n;\\
				& \tau_{\mc{D}}(\delta_i)=\maxf(\tau_{\mc{D}}(\delta_i^{k_i}),\tau_{\mc{D}}(\delta_i^{p_i})),&&1\leq i\leq m.
			\end{align*}
		\item  If $\mc{D}$ ends with
		\[
			\ax{\Gamma,\forall x\vphi^k,\vphi(t)\Ra\Delta}
				\uinf{\forall x \vphi^l,\Gamma\Ra\Delta}
			\DisplayProof
		\]
			 then $\tau_\mc{D}(\forall x \vphi^l)=\maxf(\tau_\mc{D}(\forall x\vphi^k),\tau_\mc{D}(\vphi(t)))$ and the
$\T$-complexity of the formulas in $\Gamma,\Delta$ is unchanged. 
	\item In an application of $\text{\sc (cut)}$, the $\T$-complexity of the formulas in the conclusion of the rule is treated as in case \ref{case:tccr} above.  
	\end{enumeratei}
	
	Finally, the \emph{$\tau$-complexity} of an $\lpc$-proof $\mc{D}$ is the maximum of the $\T$-complexities for the formulas occurring in it.
\end{dfn}

In what follows, it will be convenient to keep track of all derivation measure in a more compact notation. 

\begin{notation}
	We write:
		\begin{itemize-}
			\item $\lpc\sststile{m,k}{n}\Gamma\Ra \Delta$ for `the sequent $\Gamma \Ra \Delta$ has a proof in $\lpc$ with length $\leq n$, cut-rank $\leq m$, and $\T$-complexity $\leq k$'.
			\item $\lpc\sststile{}{n}\Gamma \Ra \Delta$ for  `there are $m,k$ such that $\lpc\sststile{m,k}{n}\Gamma\Ra \Delta$', and $\lpc\sststile{}{}\Gamma \Ra \Delta$ for `there is $n$ such that $\lpc\sststile{}{n}\Gamma\Ra \Delta$'.
			\item We will omit, when it's clear from the context, reference to the background system and write $\sststile{m,k}{n}$ instead of $\lpc\sststile{m,k}{n}$.
			\item We will occasionally also need to refer to the truth complexity of a single formula in a sequent as well. We will keep reference to the proof implicit, and write $\presup{k}\vphi$ for `the occurrence of $\vphi$ has truth complexity $k$ in the given derivation'. 
		\end{itemize-}
\end{notation}

The next lemma states the monotonicity of some of our measures (length and $\T$-complexity), some basic properties of $\bot$ and $\top$ in derivations, and the fully structural nature of $\lpc$ when formulas of the base language are at stake. Their proofs follow almost immediately from the definition of $\lpc \sststile{m,k}{n}$ (monotonicity), or by straightforward inductions on the length of the proof in $\lpc$.  
\begin{lemma}\hfill\label{lem:topbot}
		\begin{enumeratei}
			\item If $\lpc\sststile{m,k}{n}\Gamma \Ra \Delta$, and $k\leq k_0$ and $n\leq n_0$, then $\lpc\sststile{m,k_0}{n_0}\Gamma \Ra \Delta$. 
			\item If $\lpc\sststile{m,k}{n}\top,\Gamma \Ra \Delta$, then $\lpc\sststile{m,k}{n}\Gamma\Ra \Delta$ and the $\T$-complexity of the formulas in the contexts is unchanged. 
			\item If $\lpc\sststile{m,k}{n}\Gamma \Ra \Delta,\bot$, then $\lpc\sststile{m,k}{n}\Gamma\Ra \Delta$ and the $\T$-complexity of the formulas in the contexts is unchanged.  
			\item $\lpc \sststile{}{}\Gamma,\vphi\Ra\vphi,\Delta$ for all $\vphi\in \mc{L}$. 
		\end{enumeratei}
	\end{lemma}

The usual substitution and weakening lemmata hold for $\lpc$. Crucially for our purposes, they do not entail any increase in the $\T$-complexity of the derivation. In the case of weakening, this essentially relies on the fact that, by Definition \ref{dfn:tcompl}\ref{case:weatc}, side formulas in initial sequents have minimal $\T$-complexity. 
\begin{lemma}[Substitution, Weakening]\label{lem:invlog}
		\hfill
		\begin{enumeratei}
			\item If $\lpc\sststile{m,k}{n} \Gamma \Ra\Delta$, then $\lpc\sststile{m,k}{n}\Gamma(t/x)\Ra\Delta(t/x)$, where $t$ does not contain variables employed in applications of \textup{({\sc $\forall$r})} in the proof of $\Gamma\Ra \Delta$. The $\T$-complexity of all formulas in $\Gamma,\Delta$ is unchanged by the substitution. 
			\item If $\lpc \sststile{m,k}{n}\Gamma\Ra \Delta$, then $\lpc \sststile{m,k}{n}\Gamma,\Theta\Ra \Delta,\Lambda$ such that formulas in $\Theta, \Lambda$ have minimal complexity. Moreover, the $\T$-complexity of each formula in $\Gamma,\Delta$ is unchanged. 
		\end{enumeratei}
\end{lemma}

The next lemma contains the key property that differentiates $\lpc$ from other nonclassical and substructural approaches (cf remark \ref{rem:invert} below). Crucially, it states that truth rules are invertible in a way that does not increase neither the length nor the $\T$-complexity of the derivation. In particular, when truth ascriptions have non-zero $\T$-complexity, inversion actually \emph{reduces} their truth complexity. This property is essential for establishing the admissibility of contraction in $\lpc$ and therefore cut-elimination.

\begin{lemma}[Invertibility of $\lpc$-rules]\label{lem:invlog}
	\hfill
	\begin{enumeratei}
		\item\label{invuno} If $\lpc \sststile{m,k}{n}\Gamma,\T\corn{\vphi}\Ra \Delta$, then $\lpc\sststile{m,k}{n} \Gamma,\vphi\Ra \Delta$, with
			\begin{align*}
				&\tau(\vphi)\leq \tau(\T\corn{\vphi}),\text{ if $\tau(\T\corn{\vphi})=0$},\\
				&\tau(\vphi)<\tau(\T\corn{\vphi}),\text{ if $\tau(\T\corn{\vphi})>0$},
			\end{align*}
			and in which the $\T$-complexity in the side formulas does not increase. 
			
			A symmetric claim holds when $\lpc \sststile{m,k}{n}\Gamma\Ra\T\corn{\vphi}, \Delta$.
		\item\label{invtre} If $\lpc \sststile{m,k}{n} \Gamma,\neg\vphi \Ra\Delta$, then $\lpc \sststile{m,k}{n} \Gamma\Ra\vphi,\Delta$ with $\tau(\vphi)\leq\tau(\neg\vphi)$ and in which the $\tau$-complexity of the side formulas does not increase. 
		
		A symmetric claim holds when $\lpc \sststile{m,k}{n} \Gamma \Ra, \neg\vphi,\Delta$.
		%
	\item \label{invcinque} If $\lpc \sststile{m,k}{n} \Gamma,\vphi\land\psi\Ra\Delta$, then $\lpc \sststile{m,k}{n}\Gamma,\vphi,\psi\Ra \Delta$ with $\tau(\vphi),\tau(\psi)\leq \tau(\vphi\land\psi)$ and in which the $\tau$-complexity of the side formulas does not increase.
		%
	\item \label{invsei} If $\lpc \sststile{m,k}{n} \Gamma\Ra\vphi\land\psi,\Delta$, then $\lpc \sststile{m,k}{n}\Gamma\Ra \Delta,\vphi$ and $\lpc \sststile{m,k}{n}\Gamma\Ra \Delta,\psi$ with $\tau(\vphi),\tau(\psi)\leq \tau(\vphi\land\psi)$ and in which the complexity of the side formulas is no greater than their $\tau$-maximal occurrence in the premisses.
		%
	\item \label{invsette} If $\lpc \sststile{m,k}{n} \Gamma\Ra\Delta,\forall x\vphi$, then $\lpc \sststile{m,k}{n} \Gamma\Ra\Delta,\vphi(y)$, for  any $y$ not free in $\Gamma,\Delta,\forall x\vphi$, with $\tau(\vphi(y))\leq\tau(\forall x\vphi)$ and in which the complexity of the side formulas does not increase.
	\end{enumeratei}
\end{lemma}

\begin{proof}
	We show \ref{invuno} by induction on $n$. The other cases are easier.
	
	If $ \sststile{}{0}\Gamma,\T\corn{\vphi}\Ra \Delta$ -- i.e.~$\Gamma,\T\corn{\vphi}\Ra\Delta$ is an axiom --, then $\tau(\T\corn{\vphi})=0$. Therefore, also $ \sststile{}{0}\Gamma,\vphi\Ra\Delta$ and $\tau(\vphi)\leq \tau(\T\corn{\vphi})$. 
	
	If $\sststile{m,k}{n}\Gamma,\T\corn{\vphi}\Ra \Delta$ with $n>0$, then $\T\corn{\vphi}$ might be principal or not in the last inference.  If it's principal, we have 
	\begin{align*}
		 &\lpc \sststile{m,k_0}{n_0} \Gamma ,\presup{p_0}{\vphi}\Ra \Delta&&\text{$n_0<n,p_0<k,k_0\leq k$.}
	\end{align*}
	(recall that $\presup{p_0}{\vphi}$ signifies: $\tau(\vphi)=p_0$). The claim is then obtained by monotonicity (Lemma \ref{lem:topbot}(i)).
	 
	 If $\T\corn{\vphi}$ is not principal, let's suppose -- to consider one of the crucial cases -- that the last inference is an application of ({\sc $\T$r}). We then have:
	\begin{align*}
		&\sststile{m_0,k_0}{n_0}\Gamma,\presup{p_0}\T\corn{\vphi}\Ra\Delta_0,\presup{p_1}\psi, &&\Delta=\Delta_0,\T\corn{\psi}, n_0<n,\\
		&&&m_0=m,k_0\leq k,p_0\leq k,p_1<k.
	\end{align*}
	By the induction hypothesis, $ \sststile{m_0,k_0}{n_0} \Gamma,\presup{p_2}\vphi\Ra \Delta_0,\presup{p_1}\psi$, with $p_2\leq p_0$, and therefore, by {\sc ($\T$r)}, 
	\begin{align*}
		& \sststile{m,k}{n} \Gamma,\presup{p_2}\vphi\Ra \Delta_0,\presup{p_3}\T\corn{\psi}&& p_3\leq k. 
	\end{align*}
	The remaining cases for this subcase are similarly obtained by induction hypothesis.
	
	
\end{proof}

\begin{remark}\label{rem:invert}
It is important to observe that, in the presence of initial sequents admitting arbitrary atomic formulas of $\lt$, the inversion strategy considered above will not go through. For instance, the derivability of a sequent of the form $\Gamma,\T\corn{\vphi}\Ra\T\corn{\vphi},\Delta$ does not guarantee, for instance, the derivability of a sequent $\Gamma,\T\corn{\vphi}\Ra\vphi,\Delta$ with $\tau(\vphi)\leq \tau(\T\corn{\vphi})$. 
\end{remark}

The absence of explicit contraction -- either as a rule or by the assumption of finite sets in sequents -- is especially welcome when reasoning with measures such as the $\T$-complexity, because it may prove to be difficult to track the $\T$-complexity of each formula in a derivation if it is explicitly allowed to merge with the $\T$-complexity other occurrences of the same formula in proofs. However, as it is shown in the next lemma, contraction is an admissible rule in $\lpc$. 

\begin{lemma}[$\tau$-admissibility of contraction]\label{lem:admcon}
	\hfill
		\begin{enumeratei}
			\item \label{lem:contuno} If $\lpc\sststile{m,k}{n}\Gamma,\vphi^{p_0},\vphi^{p_1}\Ra\Delta$, then $\lpc\sststile{m,k}{n}\Gamma,\vphi\Ra\Delta$ with  $\tau(\vphi)\leq \maxf(\tau(\vphi^{p_0}),\tau(\vphi^{p_1}))$ and in which the complexity of the side formulas does not increase.
			\item \label{lem:contdue} If $\lpc\sststile{m,k}{n}\Gamma\Ra\vphi^{p_0},\vphi^{p_1},\Delta$, then $\lpc\sststile{m,k}{n}\Gamma\Ra\vphi,\Delta$ with  $\tau(\vphi)\leq \maxf(\tau(\vphi^{p_0}),\tau(\vphi^{p_1}))$ and in which the complexity of the side formulas does not increase.
		\end{enumeratei}
\end{lemma}

\begin{proof}
	\ref{lem:contuno} and \ref{lem:contdue} are proved \emph{simultaneously} by induction on $n$. Let us focus on  \ref{lem:contuno}.
	
	If $\;\sststile{}{0}\Gamma,\vphi^{p_0},\vphi^{p_1}\Ra\Delta$, then in each case $\tau(\vphi^{p_0})=\tau(\vphi^{p_1})=0$ and we have $\;\sststile{}{0}\Gamma,\vphi\Ra\Delta$ in which all formulas have $\T$-complexity $0$.  
	%
		%
		If $\sststile{}{l+1}\Gamma,\vphi^{p_0},\vphi^{p_1}\Ra\Delta$ and neither $\vphi^{p_0}$ nor $\vphi^{p_1}$ are principal in the last inference, then $\sststile{}{l+1}\Gamma,\vphi\Ra\Delta$ -- with the expected $\T$-complexities -- by induction hypothesis and possibly monotonicity.
		
		It remains the case in which $\sststile{}{l+1}\Gamma,\vphi^{p_0},\vphi^{p_1}\Ra\Delta$   and one of $\vphi^{p_0}$ or $\vphi^{p_1}$ is principal in the last inference. As an example, I treat the crucial case in which $\vphi$ is $\T\corn{\psi}$. By assumption, 
		\[
			\sststile{}{l} \Gamma,\psi^{p_{00}},\T\corn{\psi}^{p_1}\Ra\Delta
		\]
		with $\tau(\psi^{p_{00}})< \tau(\T\corn{\psi}^{p_0})\leq k$.\footnote{Of course, strictly speaking, inversion may not provide a copy of the proof in which the structure of the occurrences given by the superscripts is preserved. However, since the only relevant detail here is to distinguish between the two occurrences `to be contracted', we keep the same index for the same formulas before and after the application of inversion. }   By inversion, we have that 
		\[
			\sststile{}{l} \Gamma,\psi^{p_{00}},{\psi}^{p_{10}}\Ra\Delta.
		\]
		It can then be that $\tau({\psi}^{p_{10}})=\tau(\T\corn{\psi}^{p_1})=0$, or $\tau({\psi}^{p_{10}})<\tau(\T\corn{\psi}^{p_1})$. In both cases, we obtain 
		\[
			\sststile{}{l+1}\Gamma,\T\corn{\psi}\Ra \Delta
		\]
		with $\tau(\T\corn{\psi})\leq \maxf(\tau(\T\corn{\psi}^{p_0}),\tau(\T\corn{\psi}^{p_1}))$. It is crucial to observe that without the strong invertibility property expressed by lemma \ref{lem:invlog}(i) -- which in turn relies on the restriction of initial sequents --, one would  not be able to establish this case. In particular, if $\tau(\T\corn{\psi}^{p_1})=k> \tau(\T\corn{\psi}^{p_0})$, without the special invertibility property of Lemma \ref{lem:invlog}(i) one would not be able to complete the proof.

It is also worth noticing that the formulation of {\sc ($\forall$l)} and its associated $\T$-complexity renders the case of \ref{lem:contuno} in which one of the $\vphi$'s  is principal in the last inference and of the form $\forall x \vphi$ straightforward. Also, the simultaneous induction is especially required in the case in which the last inference is an application of ({\sc $\neg$l}) to $\vphi^{p_0}$ or $\vphi^{p_1}$ -- and symmetrically  for  \ref{lem:contdue} and ({\sc $\neg$r}).
		
\end{proof}

	The reduction lemma can now be proved in a fairly standard way. We let $(\alpha_1,\ldots,\alpha_m)\prec(\beta_1,\ldots,\beta_n)$ if $\alpha_i<\beta_i$ $(i=1,\ldots,n)$,  and for all $j< i$, $\alpha_j=\beta_j$.

	\begin{lemma}[Reduction]\label{lem:redlpt}
		If $\lpc\sststile{m,k}{n_0}\Gamma\Ra \Delta, \vphi^{l_0}$ and $\lpc\sststile{m,k}{n_1}\vphi^{l_1},\Gamma\Ra \Delta$, then $\lpc\sststile{m,k}{n_0+n_1}\Gamma\Ra \Delta$. In this latter sequent, the occurrences of formulas have $\T$-complexity no greater than the maximum of their corresponding occurrences in the assumptions of the claim.
	\end{lemma}
	
	\begin{proof}
		 The proof is by multiple, complete induction on $(l,m,n_0+n_1)$, with $l=\maxf(\tau(\vphi^{l_0}),\tau(\vphi^{l_1}))$. Our induction hypothesis is thus:
		\begin{equation}\label{eq:reduih}
			 \text{$\sststile{m',k}{n_0'}\Gamma\Ra \Delta, \psi^{l_0'}$ and $\lpc\sststile{m',k}{n_1'}\psi^{l_1'},\Gamma\Ra \Delta$ entail $\sststile{m',k}{n_0'+n_1'}\Gamma\Ra \Delta$},
		\end{equation} 
		for $|{\psi}|\leq m'$, $l'=\maxf(\tau(\psi^{l_0'}),\tau(\psi^{l_1'}))$, and $(l',m',n_0'+n_1')\prec(l,m,n_0+n_1)$. We only focus on cases in which $\T$-complexity plays a crucial role. The rest is standard. 
		 
		 If one of $\Gamma\Ra \Delta, \vphi^{l_0}$ or $\Gamma\Ra \Delta, \vphi^{l_1}$ is an axiom, one has to distinguish different subcases: If $\vphi^{l_0}$ or $\vphi^{l_1}$ are principal, then depending on whether $\vphi$ is $\bot$, $\top$, or atomic, we employ Lemma \ref{lem:topbot}(i) (in the former cases), or Lemma \ref{lem:admcon}(i). If neither of $\vphi^{l_0}$ and $\vphi^{l_1}$ is principal, then $\Gamma \Ra \Delta$ is already an axiom with minimal $\T$-complexity.  

		 Suppose now that none  of $\Gamma\Ra \Delta, \vphi^{l_0}$ or $\vphi^{l_1},\Gamma\Ra \Delta$ are axioms, but $\vphi$ is not principal in the last inference of one of their derivations, for instance the derivation of $\Gamma\Ra \Delta, \vphi^{l_0}$. In such cases, the strategy is analogous for all rules. Let's consider the case of ({\sc $\T$l}) as an example; that is, the case in which one has
		 \begin{align*}
		 	&\sststile{m,k}{n_0} \Gamma_0,\presup{p}\T\corn{\psi}\Ra \Delta,\vphi 
				&&\sststile{m,k}{n_1}\vphi,\Gamma_0,\T\corn{\psi} \Ra \Delta
		 \end{align*}
		 and the leftmost claim is obtained by ({\sc $\T$l}) from
		 \[
		 	\sststile{m,k_0}{n_{00}} \Gamma_0, \presup{p_0}\psi\Ra \Delta,\vphi
		 \]
		 with $p=p_0+1\leq k,n_{00}<n_0$ and $k_0\leq k$, and $\Gamma\equiv\Gamma_0,\T\corn{\psi}$. By the weakening lemma, we then obtain
		 \begin{align*}
		 	& \sststile{m,k_0}{n_{00}} \Gamma_0,\presup{0}\T\corn{\psi}, \psi\Ra \Delta,\vphi
				&&\sststile{m,k}{n_1}\vphi,\Gamma_0,\presup{p_1}\T\corn{\psi},\presup{0}\psi \Ra \Delta
		 \end{align*}
		with $p_1\leq k$. Since $n_{00}+n_1<n_0+n_1$, the induction hypothesis yields:
		\[
			\sststile{m,k}{n_{00}+n_1} \Gamma_0,\presup{p_1}\T\corn{\psi},\presup{p_0}\psi\Ra \Delta.
		\]
		 By applying ({\sc $\T$l}) and lemma \ref{lem:admcon}, one obtains that
		\[
			\sststile{m,k}{n_{00}+1+n_1}\Gamma_0,\T\corn{\psi}\Ra \Delta.
		\]
		This, however, yields the desired claim since $n_{00}+1+n_1\leq n_0+n_1$ and $\tau(\T\corn{\psi})=\maxf(p,p_1) $.
		
		We are left with the case in which both $\vphi^{l_0}$ and $\vphi^{l_1}$ are principal in the last inferences of the relevant derivations. Here the crucial case in which  $\vphi\equiv \T\corn{\psi}$ follows directly by the main induction hypothesis, since if our premisses are obtained via applications of the truth rules from
		\begin{align*}
			&\sststile{m,k_0}{n_{00}} \Gamma \Ra \Delta,\psi^{l_{00}}&& \sststile{m,k_1}{n_{10}} \psi^{l_{10}},\Gamma \Ra \Delta
		\end{align*}
		with $\tau(\psi^{l_{10}}),\tau(\psi^{l_{00}})< l$, the induction hypothesis and the monotonicity properties of $\lpc\sststile{}{}$ immediately yield $\sststile{m,k}{n_0+n_1}\Gamma \Ra \Delta$ with the correct $\T$-complexities in $\Gamma,\Delta$
		
		It is worth noting that the case in which $\vphi\equiv \forall x\psi$ is treated standardly as well but one has first to get rid of the universal quantifier in the premise of {\sc ($\forall$l)}. This involves an essential application of the substitution lemma that, as we know, leaves $\T$-complexities unchanged. 
	\end{proof}

	As is it clear from the Reduction Lemma, we obtain a cut-elimination theorem with standard hyper-exponential upper bounds. 
	
	\begin{corollary}
		If $\lpc \sststile{m,k}{n}\Gamma\Ra \Delta$, then $\lpc\sststile{0,k}{2^n_m}\Gamma \Ra\Delta$. 
	\end{corollary}
	Cut-elimination obviously entails the consistency of $\lpc$, defined for instance as the non-derivability of the empty sequent in $\lpc$. This may be considered to be a nice feature of $\lpc$ \emph{qua} theory of disquotational truth, as its consistency does not require more substantial notions of mathematical truth such as the ones involved in model-theoretic consistency proofs. However, often the presence of nice models -- even if interpreted in a purely instrumental way -- is a sign of the conceptual richness of one's truth predicate. We will see (section \ref{sec:seminf}) that $\lpc$ also features nice models.

	
	\section{Extension with arithmetical axioms}\label{sec:ariext}
	

The cut elimination above can be easily extended to induction-free, arithmetical base theories. For definiteness, we choose our base arithmetical theory to be Robinson's ${\rm Q}$. However, what is relevant for our discussion is the \emph{geometric} nature of such arithmetical axioms. We adapt to our setting the approach to the proof-theory of geometric rules investigated by \cite{nevo11}. Since the main structural lemmata have been introduced, this mainly involves checking that Negri and Von Plato's extension with geometric axioms interacts well with the truth rules and in particular with the notion of $\T$-complexity and its properties.

In this section we work with the language $\lnat$ of arithmetic. For definiteness, we assume the language of arithmetic is specified by the signature $\{0,{\rm S},+,\times\}$  and let $\ltnat:=\lnat\cup\{\T\}$. We assume a standard G\"odel numbering of $\lt$ and write $\#e$ for the G\"odel number of the $\lt$-expression $e$ and $\corn{e}$ for the corresponding numeral. Numerals are defined as: $\ovl{0}:=0$ and $\ovl{n+1}={\rm S}\ovl{n}$. 

The axioms of Robinson's arithmetic ${\rm Q}$ are the universal closures of the following $\lnat$-formulas:
	\begin{align*}
		  &\neg 0= {\rm S}(x),&&{\rm S}(x)={\rm S}(y)\ra x=y,\\
		&x=0\vee \exists y(x={\rm S}(y)),&& x+0=x,\\
		& x+{\rm S}(y)={\rm S}(x+y),&& x\times 0=0,\\
		& x\times {\rm S}(y)=(x\times y)+x.
	\end{align*}

As indicated in \cite{nevo11}, a ${\rm G3}$-version of ${\rm Q}$ -- equivalent to the axiom based system given above -- can be defined. In the present context, it will play the role of the base theory of our theory of truth, in that it provides us with some explicit machinery for naming sentences of our language. Unlike what is done in the previous section, we will simultaneously define derivations in our base system and the relevant measures by means of the relation $\qg\sststile{m,k}{n}$. We will include a parameter for the $\T$-complexity in this definition to allow for straightforward extensions, although of course if one focuses on purely arithmetical derivations the $\T$-complexity of the proof is always $0$.  
\begin{dfn}[$\qg$]\label{dfn:qg}
	$\qg$ extends the logic of $\lpc$ formulated in $\lt$ --  together with a restriction of {\sc (ref)} to atomic formulas of $\lnat$ and by omitting $(\bot)$ and $(\top)$ -- with the following rules
	
		\begin{itemize}\setlength\itemsep{1em}
			\item[$(=1)$] If $\qg\sststile{m,k}{n_0}\Gamma, t=t \Ra \Delta$, then $\qg\sststile{m,k}{n} \Gamma \Ra \Delta$, with $n_0< n$. 
			\item[$(=2)$] If $\qg\sststile{m,k}{n_0} s=t,\vphi(s),\vphi(t),\Gamma\Ra\Delta$, then $\sststile{m,k}{n}s=t,\vphi(t),\Gamma \Ra \Delta$, with $\vphi(v)$ an atomic formula of $\lnat$ and $n_0<n$.
			\item[$({\rm Q^g}1)$]  $\qg\sststile{m,k}{n}\Gamma,{\rm S}x=0 \Ra\Delta$ for any $n,m,k$.
			\item[$({\rm Q^g}2)$] If $\qg\sststile{m,k}{n_0} \Gamma,x=y,{\rm S}(x)={\rm S}(y)\Ra \Delta$, then $\qg\sststile{m,k}{n}\Gamma,{\rm S}(x)={\rm S}(y)\Ra \Delta$, with $n_0<n$.
			\item[$({\rm Q^g}3)$] If $\qg\sststile{m,k}{n_0}\Gamma,x=0\Ra \Delta$ and $\qg\sststile{m,k}{n_1}\Gamma,y={\rm S}(x)\Ra \Delta$, then $\qg\sststile{m,k}{n}\Gamma \Ra \Delta$, with $n_0,n_1<n$ and with $y\notin {\rm FV}(\Gamma,\Delta,x=0)$.
			\item[$({\rm Q^g}4)$] If $\qg\sststile{m,k}{n_0} \Gamma,x+0=x\Ra \Delta$, then $\qg\sststile{m,k}{n}\Gamma\Ra \Delta$, with $n_0<n$. 
			\item[$({\rm Q^g}5)$] If $\qg\sststile{m,k}{n_0}\Gamma,x+{\rm S}(y)={\rm S}(x+y)\Ra \Delta$, then $\qg\sststile{m,k}{n_0}\Gamma\Ra\Delta$, with $n_0<n$. 
			\item[$({\rm Q^g}6)$] If $\qg\sststile{m,k}{n_0} \Gamma,x\times 0=0\Ra \Delta$, then $\qg\sststile{m,k}{n}\Gamma\Ra \Delta$, with $n_0<n$. 
			\item[$({\rm Q^g}7)$] If $\qg\sststile{m,k}{n_0}\Gamma,x\times {\rm S}(y)=(x\times y)+x\Ra \Delta$, then $\qg\sststile{m,k}{n}\Gamma\Ra\Delta$, with $n_0<n$. 
		\end{itemize}
\end{dfn}
\noindent In $({\rm Q^g}3)$, $y$ acts as an eigenvariable, because it is intended to be playing the role of an existentially quantifiable variable. 

As before, by a straightforward induction on the length of the proof in $\qg$, we can show that, as far as formulas of $\lnat$ are concerned, reflexivity holds for them.  
The next lemma states that, as desired, $\qg$ and ${\rm Q}$ prove the same theorems.
	\begin{lemma}
	 	${\rm Q}\vdash \bigwedge \Gamma \ra \bigvee \Delta$ if and only if $\qg\vdash \;\Gamma \Ra \Delta$. 
	\end{lemma}
	
%
	
	The system $\lptn$ is obtained by extending $\qg$ with fully disquotational truth. The truth rules are only notational variations of {\sc ($\T$l)} and {\sc ($\T$r)}. 
		\begin{dfn}
		The relation $\lptn \sststile{m,k}{n}$ is defined by means of the direct analogues of clauses  $(=1)$-$(\qg7)$ from Definition \ref{dfn:qg} plus:
		\begin{itemize}
			\item[$(\T\text{{\sc r}}^\nat)$] If $\lptn\sststile{m,k_0}{n_0}\Gamma \Ra\vphi,\Delta$, then $\lptn\sststile{m,k}{n}\Gamma\Ra \T \ovl{l},\Delta$, with $n_0<n, k_0\leq k$, $l=\#\vphi$ with $\vphi$ a sentence of $\lt$, $\tau(\T\ovl{l})=\tau(\vphi)+1$, and the $\T$-complexities of the side formulas are unchanged. 
			\item[$(\T\text{{\sc l}}^\nat)$] If $\lptn\sststile{m,k_0}{n_0}\Gamma,\vphi \Ra\Delta$, then $\lptn\sststile{m,kk}{n} \Gamma,\T \ovl{l}\Ra\Delta$, with $n_0<n, k_0\leq k$, $l=\#\vphi$ with $\vphi$ a sentence of $\lt$, $\tau(\T\ovl{l})=\tau(\vphi)+1$, and the $\T$-complexities of the side formulas are unchanged. 
		\end{itemize}
			\end{dfn}
	\begin{remark}\label{rem:purevar}
		In the rest of the section, we assume that so-called pure variable convention. That is, free and bound variables are always distinct in proofs, and that the eigenvariables of applications of $({\rm Q^g}3)$ in proofs are distinct. 
	\end{remark}
			
	As before, the identity axioms hold unrestrictedly for sentences of $\lnat$, so we have
	\begin{equation}\label{eq:refq}
		\text{$\lptn \vdash \Gamma,\vphi\Ra \vphi,\Delta$ for all $\vphi\in \lnat$.} 
	\end{equation}
	
	 The substitution lemma for $\lptn$ -- compared with its analogue in the previous section -- needs a little extra care in dealing with the variables of the geometric rules. Essentially, in the required induction on the length of the proof in $\lptn$, the cases of {\sc ($\forall$r)} and $({\rm Q^g}3)$ require the eigenvariables not to occur in the \emph{substituens}. Similarly, in the weakening lemma one only needs to be careful that the weakened formulas do not contain variables that may appear in geometric rules. In such cases the substitution lemma can be employed. $\T$-complexities are handled in precisely the same way as before. 
	
	\begin{lemma}[Substitution, Weakening]\hfill
		\begin{enumeratei}
			\item If $\lptn\sststile{m,k}{n}$, then $\lptn\sststile{m,k}{n} \Gamma(t/x)\Ra \Delta(t/x)$ where $t$ is free for $x$ in $\Gamma,\Delta$ and it does not contain any eigenvariables employed in applications of {\sc ($\forall$r)}, as well as variables employed  ${\rm Q^g}$-rules. The substitution does not change the $\T$-complexity of the formulas occurring in $\Gamma,\Delta$. 
			\item If $\lptn \sststile{m,k}{n}\Gamma\Ra \Delta$, then $\lptn \sststile{m,k}{n}\Gamma,\Theta\Ra \Delta,\Lambda$ with $\Theta$ and $\Lambda$ not containing variables appearing in geometric rules and whose formulas have minimal $\T$-complexity. Moreover, the $\T$-complexity of each formula in $\Gamma,\Delta$ is unchanged. 
		\end{enumeratei}
	\end{lemma}
	The invertibility lemma also proceeds with minor variations. Crucially, the kind of $\tau$-invertibility for the truth rules involved in lemma \ref{lem:invlog}(i) is preserved. To prove an analogue of Lemma \ref{lem:invlog}(v), one employs Remark \ref{rem:purevar} to ensure that if the last inference involves a geometric rule such as $({\rm Q^g}3)$, the role of the eigenvariable in the geometric rule is preserved.
	\begin{lemma}[Inversion]
		The propositional logical rules of $\lptn$ are $\tau$-invertible in the way described by Lemma  \ref{lem:invlog}(ii)-(iv). Moreover:
		\begin{enumeratei}
			\item If $\lptn \sststile{m,k}{n}\Gamma,\T\ovl{l}\Ra \Delta$ with $l=\#\vphi$, then $\lpc\sststile{m,k}{n} \Gamma,\vphi\Ra \Delta$, with
			\begin{align*}
				&\tau(\vphi)\leq \tau(\T\ovl{l}),\text{ if $\tau(\T\ovl{l})=0$},\\
				&\tau(\vphi)<\tau(\T\ovl{l}),\text{ if $\tau(\T\ovl{l})>0$},
			\end{align*}
			and with unchanged $\T$-complexity in the side formulas. 
			
			A symmetric claim holds when $\lpc \sststile{m,k}{n}\Gamma\Ra\T\corn{\ovl{l}}, \Delta$ with $l=\#\vphi$.
			
			\item  If $\lptn \sststile{m,k}{n} \Gamma\Ra\Delta,\forall x\vphi$, then $\lptn \sststile{m,k}{n} \Gamma\Ra\Delta,\vphi(y)$, for  any $y$ not free in $\Gamma,\Delta,\forall x\vphi$ and not among the variables of geometric rules, with $\tau(\vphi(y))\leq\tau(\forall x\vphi)$ and in which the complexity of the side formulas does not increase.
		\end{enumeratei}	
	\end{lemma}
	The previous lemmata makes it possible to extend in a straightforward way the $\tau$-admissibility of contraction to $\lptn$.
	\begin{lemma}
		 If $\lptn\sststile{m,k}{n}\Gamma,\vphi^{k_0},\vphi^{k_1}\Ra\Delta$, then $\lptn\sststile{m,k}{n}\Gamma,\vphi\Ra\Delta$ with with $\tau(\vphi)\leq \maxf(\tau(\vphi^{k_0}),\tau(\vphi^{k_1}))$ and in which the complexity of the side formulas does not increase.
			A symmetric claim holds for when the formulas to be contracted appear on the consequent.
	\end{lemma}	
	
	With these lemmata at hand, we are then able to prove a reduction lemma in the same vein as the previous section. Noticeably, the interaction between truth, identity, and arithmetical rules is particularly smooth because truth rules only apply to closed terms naming sentences, and therefore no extra-care with variables is needed to deal with cases in which the elimination of a cut on a non-principal truth ascription is obtained by performing the cut on the premisses of a geometric rule. The cut-elimination procedure in the presence of geometric rules does not change the hyperexponential upper-bound. 
	
	\begin{corollary}\label{cor:cuello}
		Cut is eliminable in $\lptn$. 
	\end{corollary}
	
	 The method outlined in this section straightforwardly extends to geometric rules corresponding to the defining equations of other primitive recursive functions. One could also then strengthen the truth rules, for instance, to pointwise compositional rules such as:
	\[
		\ax{\Gamma\Ra \Delta,\vphi}
		\ax{\Gamma\Ra \Delta,\psi}
			\binf{\Gamma \Ra\Delta,\T(\ovl{l}\subdot\land\ovl{m})}
		\DisplayProof
	\]
with $\#\vphi=l,\#\psi=m$ and $\subdot\land$ the function symbol representing in $\lnat$ the syntactic operation 
	\[\#\vphi,\#\psi\mapsto \#(\vphi\land\psi).\] 

Finally, Corollary \ref{cor:cuello} and subsequent remarks clearly yields conservativity properties of the rules ($\T${\sc l}) and ($\T${\sc r}) over base theories given by geometric axioms. In fact, for $\vphi\in \lnat$, if $\lptn\vdash \,\Ra\vphi$, then there is a cut-free proof $\mc{D}$ of $\Ra\vphi$. All succedents in $\mc{D}$ must be subformulas of $\vphi$, and all formulas in the antecedents must be formulas of $\lnat$, because they are the only ones that may disappear due to geometric and identity rules. Therefore, we have:

\begin{prop}\label{prop:conse}
	$\lptn$ is a conservative extension of $\qg$. 
\end{prop}


	\section{Infinitary rules and semantics}\label{sec:seminf}
	
	In this section we first extend $\lpc$ to an infinitary system $\lptinf$, and then establish the anticipated links between $\lptinf$ and fixed-point semantics. 
	
	\subsection{Infinitary rules} It is convenient to work with an expansion of $\lnat$ with function symbols for primitive recursive functions, which we call $\lpr$. We then in turn denote with $\lprt$ the expansion of $\lpr$ with the predicate $\T$. $\lpr$ will then contain function symbols corresponding to syntactic operations on G\"odel numbers such as $\subdot\land$ above  and
	\begin{align}
		\label{eq:conite}&n\mapsto \#(\T\ovl{n}), && n,m\mapsto \#(\underbrace{\T\ulcorner\ldots\ulcorner\T}_{\text{$m$ $\T$\emph{s}}}\ovl{n}\urcorner\ldots\urcorner)\\
		\notag&n\mapsto \#\ovl{n}  &&\#\vphi(v),\#t\mapsto \#(\vphi(t/v))\\
		\notag&\#\vphi \mapsto \#(\neg\vphi)&& \#\vphi,\#v\mapsto \#(\forall v\vphi)&& \#s,\#t\mapsto \#(s=t)
	\end{align}
	We will employ, respectively, the function symbols $\subdot \T,{\rm tr},{\rm num},{\rm sub},\subdot\neg,\subdot \forall,\subdot =$ to express those operations in our language. On occasion we will make reference to a function symbol ${\rm val}$ for a recursive evaluation function for primitive recursive functions expressing the semantic evaluation function $t\mapsto t^\nat$ taking a closed term and returning its value in the standard model of $\lnat$.
	
	The infinitary system $\lptinf$ is essentially obtained by reformulating  $\lpc$ in $\lprt$, replacing basic truth and falsities with arithmetical truths and falsities, and supplementing the system with an $\omega$-rule. Later on we will also consider the language $\lnat^2$ of second-order arithmetic, extending $\lpr$ with second-order (relational) variables and quantifiers. The presence of the $\omega$-rule makes the {\it length} of derivation, as well as the associated $\T$-complexities, possibly infinite -- more precisely, a countable ordinal. In particular, the definition of $\T$-complexity  needs to be supplemented with the case in which a derivation ends with an application of the infinitary rule. This can be informally described as follows. If a derivation $\mc{D}$  ends with
\begin{eq}
	\ax{\ldots}
	\ax{\gamma^{j_{i_1}}_1,\ldots,\gamma^{j_{i_n}}_n\Ra\delta^{k_{i_1}}_1,\ldots,\delta^{k_{i_m}}_m\vphi(t_i)}
	\ax{\ldots}
		\TrinaryInfC{$\gamma_1,\ldots,\gamma_n\Ra\delta_1,\ldots,\delta_m,\forall x\vphi$}
	\DisplayProof
\end{eq}
then:
\begin{align*}
	& \tau(\gamma_k):= {\rm sup}\{\tau(\gamma^{j_{i_k}}_k) \sth i\in \omega,1\leq k\leq n\},\\
	&\tau(\delta_l):= {\rm sup}\{\tau(\delta^{k_{i_l}}_l) \sth i\in \omega,1\leq l\leq m\},\\
	&\tau(\forall x\vphi):={\rm sup}\{\tau(\vphi(t))\sth \text{ $t$ a closed term of $\lpr$}\}.
\end{align*}

	Here's the official definition of the infinitary system $\lptinf$:
	\begin{dfn}[$\lptinf$]
$\lptinf$ is obtained from $\lpc$ by:
	\begin{itemize}\setlength\itemsep{1em}
		\item  Omitting free variables.
		%
		%
		\item Replacing the axioms {\sc ($\top$)}, {\sc ($\bot$)} with\\
			\begin{itemize}
				\item[$(\mbb{T})$] $\sststile{m,\beta}{\alpha}\Gamma\Ra r=s,\Delta$  for any $\alpha,\beta,m$ and with $r^\nat=s^\nat$; \\
					 \item[$(\sc \mbb{F})$] $\sststile{m,\beta}{\alpha}\Gamma,r=s\Ra\Delta$  for any $\alpha,\beta,m$ and with $r^\nat\neq s^\nat$.
			\end{itemize}
		\item Replacing {\sc $\T$l} and {\sc $\T$r} with the more general:
		\\
		\begin{itemize}
			\item[$(\T\text{{\sc r}}^\nat)$] If $\lptn\sststile{m,\beta}{\alpha}\Gamma \Ra\vphi,\Delta$, then $\lptn\sststile{m,\delta}{\gamma}\Gamma\Ra \T t,\Delta$, with $\alpha<\gamma, \beta<\delta$, $t^\nat=\#\vphi$, $\tau(\T t)=\tau(\vphi)+1$, and the $\T$-complexities of the side formulas are unchanged. 
			\item[$(\T\text{{\sc l}}^\nat)$] If $\lptn\sststile{m,\beta}{\alpha}\Gamma,\vphi \Ra\Delta$, then $\lptn\sststile{m,\delta}{\gamma} \Gamma,\T t\Ra\Delta$, with $\alpha<\gamma, \beta<\delta$, $t^\nat=\#\vphi$, $\tau(\T t)=\tau(\vphi)+1$, and the $\T$-complexities of the side formulas are unchanged.
		\end{itemize}
%
		%
		\item Replacing {\sc ($\forall$r)} with:\\
			\begin{itemize}
				\item[$(\omega)$] If for all $t$ there are $\alpha<\gamma$ and $\beta\leq\delta$ such that $\sststile{m,\beta}{\alpha}\Gamma \Ra \vphi(t),\Delta$, then $\sststile{m,\delta}{\gamma} \Gamma \Ra \forall x\vphi,\Delta$, with $\tau(\forall x\vphi)={\rm sup}\{\tau(\vphi(t))\sth \text{$t$ a term of $\lpr$}\}$.
			\end{itemize}

	\end{itemize}
\end{dfn}

	\bigskip
	\begin{remark}\label{rem:lpi}
	\hfill
		\begin{enumeratei}
			%
			\item The general formulation of $(\T\text{{\sc r}}^\nat)$ and $(\T\text{{\sc l}}^\nat)$ is essential for the claims below. It allows transfinite iterations of applications of $\T$, which are otherwise not available, even in the presence of the $\omega$-rule. This can be easily seen by considering the function representing the rightmost operation in \eqref{eq:conite}, which we call ${\rm tr}(n,t)$. 
		\[
			\lptinf\vdash \T({\rm tr}(\ovl{n},\corn{0=0}))\;\;\text{for any $n\in \omega$.}
		\]
		The $\omega$-rule then gives us transfinite iterations of $\T$. This process, of course, carries on for further recursive ordinals by carefully choosing syntactic operations akin to ${\rm tr}(\cdot)$.
			\item As before, $\lpcinf$ proves identity sequents $\Gamma,\vphi\Ra \vphi,\Delta$ for all $\vphi\in \lpr$. 
		\end{enumeratei}
	\end{remark}

Then the cut-elimination strategy proceeds with only minor variations. We have:

\begin{lemma}\hfill \label{lem:inffac}
	\begin{enumeratei}	
		\item\label{weainf} (Weakening) If $\lptinf\sststile{m,\beta}{\alpha} \Gamma\Ra \Delta$, then $\lptinf\sststile{m,\beta}{\alpha}\Gamma_0,\Gamma\Ra \Delta_0,\Delta$ with all formulas in $\Delta_0, \Gamma_0$ featuring minimal $\T$-complexity. 
		\item \label{invinf} (Inversion) All rules shared by $\lptinf$ and $\lpc$ are length-, and $\tau$-invertible as prescribed by Lemma \ref{lem:invlog}, \ref{invuno}-\ref{invsei}. Moreover:
			\begin{quotation}
				if $\lptinf\sststile{m,\beta}{\alpha} \Gamma\Ra\Delta,\forall x\vphi$, then $\lptinf\sststile{m,\beta}{\alpha} \Gamma\Ra\Delta,\vphi(t)$, for  any  closed term $t$. In addition, $\tau(\vphi(t_i))\leq \tau(\forall x\vphi)$ and the $\T$-complexity of the formulas in the contexts of the inverted sequents is unchanged. 
			\end{quotation}
		\item Contraction is $\tau$-preserving and length-preserving admissible in $\lptinf$.
	\end{enumeratei}
\end{lemma}

In particular, the reduction lemma generalizes to ordinals in the expected way. It is obviously important to employ ordinal addition in the induction to deal with cuts on principal formulas of $\omega$-rules. 

\begin{prop}\label{prop:ceinf}
	If $\lpcinf\sststile{m,\beta}{\alpha_0}\Gamma \Ra \Delta, \vphi$ and $\lpcinf\sststile{m,\beta}{\alpha_1}\vphi,\Gamma \Ra \Delta$, then $\lpcinf\sststile{m,\beta}{\alpha_0+\alpha_1}\Gamma \Ra \Delta$. Therefore,
	\begin{center}
		If $\lptinf\sststile{m,\beta}{\alpha}\Gamma\Ra \Delta$, then $\lptinf \sststile{0,\beta}{}\Gamma\Ra \Delta$.
	\end{center}
\end{prop}

In light of Proposition \ref{prop:ceinf}, one can employ $\lpcinf$ to establish the consistency, via appropriate embeddings, of finitary extensions of $\lptn$ with induction axioms.

	\subsection{Fixed-point models}
	
	We relate the infinitary system introduced above and a fixed-point model for $\lt$. The results of this section are intended to compare to \cite{figr18} and, less directly, to \cite{can90}. We show that the form of invertibility allowed by $\lptinf$ enables one to preserve fully disquotational truth rules \emph{with context} while capturing Kripkean grounded truth. In the references mentioned, this could only be achieved by disallowing contexts altogether from the premisses of ($\T${\sc l}) and ($\T${\sc r}): this, in turn, would render such rules less obviously `truth rules', but open to other interpretations such as the ones based on provability.
	
	Consider the following formula of the language $\lnat^2$ of second-order arithmetic, where ${\rm CT}(x)$ and ${\rm Sent}_{\lprt}$ are primitive recursive predicate expressing the notions of being a closed term and a sentence of $\lprt$ respectively:
		\begin{align*}
			{\rm K}(X,x) :\lra \;& {\rm Sent}_{\lprt}(x)\;\land\\
			&\exists y\exists z({\rm Ct}_{\lprt}(y)\land {\rm Ct}_{\lprt}(z)\land x=(y\subdot = z)\land {\rm val}(y)={\rm val}(z)) \;\vee\\ 
			& \exists y\exists z({\rm Ct}_{\lprt}(y)\land {\rm Ct}_{\lprt}(z)\land x=\subdot \neg(y\subdot = z)\land {\rm val}(y)\neq{\rm val}(z)) \;\vee\\ 
			&\exists y({\rm CT}(y)\land x={\rm sub}(\corn{\T v},\corn{v},y)\land {\rm val}(y)\in X)\;\vee\\
			&\exists y({\rm CT}(y)\land x={\rm sub}(\corn{\neg\T v},\corn{v},y)\land (\subdot\neg{\rm val}(y))\in X)\;\vee\\
			&\exists y({\rm Sent}_{\lprt}(y)\land x=(\subdot\neg\subdot\neg y)\land y\in X)\;\vee\\
			&\exists y\exists z({\rm Sent}_{\lprt}(x)\land x=(y\subdot\land x)\land  y\in X\land z\in X)\;\vee\\
			& \exists y\exists z({\rm Sent}_{\lprt}(x)\land x=(\subdot \neg(y\subdot\land z))\land (\subdot\neg y)\in X \vee (\subdot\neg z)\in X)\;\vee\\
			& \exists v\exists y({\rm Sent}_{\lprt}(x)\land x=(\subdot \forall vy)\land \forall u({\rm CT}(u)\ra {\rm sub}(x,v,u)\in X))\;\vee\\
			& \exists v\exists y({\rm Sent}_{\lprt}(x)\land x=(\subdot\neg\subdot\forall vy)\land \exists u({\rm CT}(u)\land \subdot\neg{\rm sub}(x,v,u)\in X)).
		\end{align*}
		By inspection of ${\rm K}(X,x)$, it is clear that $X$ occurs positively in it, in the sense that it does not contain occurrences of $u\notin X$, and essentially so.\footnote{More precisely, this means that we can translate ${\rm K}(X,x)$ in a Tait-language and no occurrences of $u\notin X$ are present. } We define an operator ${\varPhi_{\rm K}} \colon \mbb{P} (\omega)\to \mbb{P}(\omega)$ as follows:
		\[
			{\varPhi_{\rm K}}(S):=\{ n \sth (\nat, S) \vDash {\rm K}(X,\ovl{n})\},
		\]
		where $(\nat,S)$ expresses that $S$ is used to interpret the variable $X$. Since ${\rm K}(X,x)$ is $X$-positive, ${\varPhi_{\rm K}}$ is monotone:
		\begin{equation}\label{eq:monot}
			S_0\subseteq S_1 \text{ only if } {\varPhi_{\rm K}}(S_0)\subseteq {\varPhi_{\rm K}}(S_1).
		\end{equation}
	By transfinite recursion, one then sets: 
	\begin{align*}
		&{\varPhi_{\rm K}}^{\alpha}:={\varPhi_{\rm K}}({\varPhi_{\rm K}}^{<\alpha}),\text{ with }{\varPhi_{\rm K}}^{<\alpha}:=\bigcup_{\beta<\alpha}{\varPhi_{\rm K}}^\beta.
	\end{align*}
	It is clear that ${\varPhi_{\rm K}}$ has \emph{fixed points}, i.e.~there are ordinals $\gamma$ such that ${\varPhi_{\rm K}}^{<\gamma}={\varPhi_{\rm K}}^\gamma$.\footnote{Otherwise, $\{n\sth n\in {\varPhi_{\rm K}}^\alpha\setminus {\varPhi_{\rm K}}^{<\alpha}\}$ would be a subset of $\nat$ of cardinality $\aleph_1$.} We let $\kappa_K={\rm min}\{\alpha \sth {\varPhi_{\rm K}}^{<\alpha}={\varPhi_{\rm K}}^\alpha\}$, and ${\rm I}_{{\varPhi_{\rm K}}}:={\varPhi_{\rm K}}^{<\alpha}$. ${\rm I}_{\varPhi_{\rm K}}$ is called the \emph{minimal fixed point} of ${\varPhi_{\rm K}}$. It is well-known (see e.g.~\cite[Ch.~6]{poh09}), that $\kappa_K=\omega^{\rm CK}_1$. 
	
	${\rm I}_{\varPhi_{\rm K}}$ is well-known for capturing the concept of grounded truth \cite{kri75}, because any $\#\vphi\in \mfp{\varPhi_{\rm K}}$ is either a true atomic sentence of $\lpr$ or an atomic truth appears in its `dependency' structure \cite{lei05}. We will briefly return on the connection between $\lptinf$ and grounded truth shortly.

	For $n\in \mfp{{\varPhi_{\rm K}}}$, its \emph{inductive norm} is defined as:
		\[
			|n|_{\varPhi_{\rm K}}:= {\rm min}\{ \alpha \sth n\in {\varPhi_{\rm K}}^\alpha\}.
		\]
	We also have:
		\begin{eq}
			n\in \mfp{\opk} \;\text{ iff }\; \forall X(\forall x({\rm K}(X,x)\ra x\in X)\ra n\in X),
		\end{eq}
		so $\mfp{\opk}$ is $\Pi_1^1$-definable in $\lnat^2$. As noticed by \cite{kri75} (see also \cite{bur86}), $\mfp{\opk}$ is $\Pi^1_1$-complete.\footnote{The idea of the proof: one can uniformly replace $y\in X$ by ${\rm Tr}\,{\rm sub}(u,\corn{v},{\rm num}(y))$ in $P(x,X)$ -- an arbitrary inductive definition -- to obtain $P'(x,u)$. The diagonal lemma then yields a formula $\xi(v)$ such that
		\[
			(\nat,X)\vDash \xi(v)\;\text{ iff }\;(\nat,X)\vDash P'(x,\corn{\xi(v)}).
		\]
		Finally, one shows by transfinite induction on the generation of the minimal fixed point $\mfp{P}$ that $n\in {\rm I}_P$ if and only if $(\nat,\mfp{\opk})\vDash \T\corn{\xi(\dot{n})}$.}
		
		The strict relationships between $\Pi^1_1$-sets and infinitary cut-free calculi are secured by general results \cite[\S6.6]{poh09}. To witness the link between $\lptinf$ and $\mfp{\opk}$, we establish a direct correspondence between the two frameworks. The existence of a nice semantics for $\lpc$-based systems will then immediately follow.

\begin{lemma}\label{lem:grocon}
	If $\lptinf \sststile{0,\beta}{\alpha}\Gamma \Ra \Delta$, then \emph{either} there is a $\gamma \in \Gamma$ with $|\#\neg\gamma|_{\varPhi_{\rm K}}\leq\alpha$ \emph{or} there is a $\delta\in \Delta$ with $|\#\delta|_{\varPhi_{\rm K}}\leq\alpha$.
\end{lemma}

\begin{proof}
	The proof is by transfinite induction on $\alpha\leq \omega_1^{\rm CK}$. 
	
	If $\alpha=0$, then the claim follows by definition for $(\sc \mbb{T})$ and $(\sc \mbb{F})$, or from the fact that closed atomic identities of $\lprt$ are decided by $\mfp{\varPhi}$. 
	
	If $\alpha$ is successor or limit the claim follows by inductive hypothesis by reflecting on the fact that the disjuncts in ${\rm K}$ harmonize well with the rules of $\lpcinf$. For instance, if $\Gamma \Ra \Delta$ is proved by an application of $(\T\text{{\sc r}}^\nat)$, we have 
	\[
		\sststile{0,\beta_0}{\alpha_0} \Gamma \Ra \Delta_0,\vphi
	\]
	with $\beta_0\leq \beta,\alpha_0\leq \alpha$, $\tau(\vphi)<\beta$. If some $\gamma \in \Gamma$ is such that $|\#\neg\gamma|_{\varPhi_{\rm K}}\leq \alpha_0$, or some $\delta_0\in \Delta_0$ is such that $|\#\delta_0|_{\varPhi_{\rm K}}\leq \alpha_0$, we are done by the definition of $\varPhi^{\alpha}$. Otherwise, by induction hypothesis, we have that $|\#\vphi|_{\varPhi_{\rm K}}\leq \alpha_0$ and therefore, for $t^\nat=\#\vphi$,  $|\T t|_{\varPhi_{\rm K}}\leq \alpha$. 
	
	Noticeably, even though ${\rm K}$ does not feature a full clause for negation, $\mfp{\opk}$ can still capture their behaviour in the absence of initial sequents. Suppose for instance that $\Gamma \Ra \Delta$ is such that $\Delta=\Delta_0,\neg\vphi$ and $\Gamma \Ra \Delta$ is obtained with an application of ({\sc $\neg$r}), so that 
	\[
		\sststile{0,\beta_0}{\alpha_0}\Gamma,\vphi \Ra \Delta
	\]
	with $\beta_0\leq \alpha_0<\alpha$, $\alpha_0\leq \alpha$. By induction hypothesis, either $|\#\neg\gamma|_{\varPhi_{\rm K}}=\alpha_0$ for some $\gamma\in \Gamma$, or $|\#\neg\vphi|_{\varPhi_{\rm K}}=\alpha_0$, or $|\#\delta|_{\varPhi_{\rm K}}=\alpha_0$ for $\delta\in \Delta$. In each case, we obtain the claim by definition of $\mfp{\opk}$. Similarly, if $\Gamma \Ra \Delta$ is obtained with an application of ({\sc $\neg$l}), so that 
	\[
		\sststile{0,\beta_0}{\alpha_0}\Gamma \Ra\vphi, \Delta,
	\]
	in the crucial case in which $|\#\vphi|_{\opk}\leq\alpha_0$, one has that $|\#\neg\neg\vphi|_{\opk}\leq \alpha$, as required. 

\end{proof}

Lemma \ref{lem:grocon} reveals a non-standard way of thinking about `logical' consequence -- or better, satisfiability of sequents -- in Kripke models which is intrinsic to $\lpc$ and extensions thereof. If in the customary approach -- cf.~for instance the literature stemming from \cite{haho06} -- the satisfiability of a sequent $\Gamma \Ra \Delta$ is defined as preservation of truth in fixed-point models, the notion of consequence underlying the semantics of $\lpc$ is based on the existence of appropriate determinate truth values (false in the antecedent, true in the succedent). In the terminology of \cite{coba12}, this is a tolerant-strict notion of consequence. As mentioned earlier in the paper, \cite{niro18} proposed such notion of consequence and showed that this semantics is compatible with a primitive, self-applicable predicate for consequence which fully internalizes it in the object-language. A more comprehensive study of such notion, including the formulation of a compositional theory of a truth whose $\omega$-models are exactly the fixed points of $\varPhi_{\rm K}$ above, is carried out in \cite{niro20}.

Conversely, we also have that the extension of $\mfp{\varPhi}$ can be characterized in terms of $\lptinf$ proofs. 
\begin{lemma}\label{lem:mfpsys}
 If $|\#\vphi|_{\opk}\leq \alpha $, then there is an $n\in \omega$ such that $\lptinf \sststile{0,\beta}{\alpha+n}\;\Ra \vphi$, with $\beta\leq \alpha+n$.
\end{lemma}

\begin{proof}
	The proof is again by induction on $\alpha$. 
	
	If $\alpha=0$, $\vphi$ can only be $s=t$ or $s\neq t$ for closed terms $s,t$ and $s^\nat=t^\nat$ or $s^\nat \neq t^\nat$ respectively. In the latter case, one has $\sststile{0,0}{0} \Gamma,s=t\Ra \Delta$, and therefore $\sststile{}{1}\Gamma \Ra s\neq t,\Delta$. In the former, simply $\sststile{0,0}{0} \Gamma\Ra s=t, \Delta$.
	
	If $\alpha>0$ is a limit ordinal, the claim follows directly by induction hypothesis. If $\alpha$ is a successor ordinal, one consider the different clauses in $\varPhi_{\rm K}$. The mismatch between norm and length of proof is essentially required when negated formulas are considered. For instance, if $|\#\neg \T t|_{\opk}= \alpha$, then $|\#\neg\psi|_{\opk}=\alpha_0<\alpha$ with $t^\nat=\#\psi$. The induction hypothesis then yields $\sststile{0,\beta_0}{\alpha_0+m}\; \Ra \neg \psi$ for some $m$, $\beta_0<\beta$. By the inversion Lemma \ref{lem:inffac}, we obtain $\sststile{0,\beta_0}{\alpha_0+m}\; \psi\Ra $. The claim is then obtained by $(\T\text{{\sc l}}^\nat)$, and ({\sc $\neg$r}).  
	
	Similarly, if $|\#\neg\neg \psi|_{\opk}\leq \alpha$, then $|\#\psi|_{\opk}\leq \alpha_0< \alpha$. By induction hypothesis, $\sststile{0,\beta_0}{\alpha_0+m}\Ra\psi$ with $\beta_0\leq \beta$ for some $m$. By the negation rules, $\sststile{0,\beta_0}{\alpha_0+m+2}\Ra\neg\neg\psi$.
\end{proof}
By inspection of the proof of Lemma \ref{lem:mfpsys}, one notices that the following claim also holds, yielding a symmetric picture to the one depicted by Lemma \ref{lem:grocon}:
\begin{equation}
	\text{if $|\#\neg\vphi|_{\opk}\leq \alpha $, then there is an $n\in \omega$ such that $\lptinf \sststile{0,\beta}{\alpha+n}\;\vphi\Ra$, with $\beta\leq \alpha+n$.}
\end{equation}
\begin{corollary}\hfill
	\begin{enumeratei}
		\item $|\#\vphi|\in \mfp{\opk}$ if and only if $\lptinf\sststile{0,\beta}{\alpha}\;\Ra \vphi$ for some $\beta\leq\alpha<\omega_1^{\rm CK}$.
		\item $|\#\neg\vphi|\in \mfp{\opk}$ if and only if $\lptinf\sststile{0,\beta}{\alpha}\;\vphi\Ra$ for some $\beta\leq\alpha<\omega_1^{\rm CK}$.
		\end{enumeratei}
\end{corollary}

It is clear that, since we can straightforwardly embed $\lptn$ -- but, as mentioned in the previous subsection also an extensions of $\lptn$ with full $\lt$-induction -- in $\lptinf$, the results above also amount to soundness proofs, with respect to fixed-point semantics, of our systems modulo the notion of consequence relation specified by Lemma \ref{lem:grocon}.

\section{Conclusion}


The focus of this paper is on the structural properties of theories of fully disquotational truth with restricted initial sequents. If one finds the framework appealing for the basic logical properties presented here, there are certainly further philosophical and technical questions to be investigated. 

The kind of reasoning available in theories such as $\lpc$ and extensions thereof displays peculiar properties. First of all, the rules of inference available are entirely classical. Moreover, the systems reveal a special relationships occurring between truth ascriptions and the underlying base language which is not available in alternative formal systems for transparent truth. Philosophers often explain grounded truth in terms of a form of supervenience of truth on the non-truth-theoretic world (cf.~for instance \cite{lei05}).  Theories in the style of $\lpc$ seem to capture this idea in a particularly strong way. Essentially, the absence of initial sequents featuring the truth predicate blocks the possibility of reasoning hypothetically with arbitrary truth ascriptions. Only formulas of the base language can be freely assumed in reasoning -- cf.~Lemma \ref{lem:topbot}(iv), \eqref{eq:refq}, Remark \ref{rem:lpi}(ii).  Semantically speaking, in the context fully structural approaches, one can perform hypothetical reasoning also by employing sentences that may not have a determinate truth value.\footnote{Philosophically, this may be dealt with, for instance, as in \cite{kri75}, by applying Strawson's analysis that the hypothesis of a truth ascription should be understood as an attempt to make a claim, to express a proposition.} In the present framework this is ruled out, and hypothetical reasoning is only available for sentences that are determinately true or false, such as sentences of the base language. This does not amount to say that for \emph{no} sentence containing the truth predicate some form of hypothetical reasoning is available. The framework automatically enables one to iterate the truth predicate over sentences that are `grounded'. For instance, the inference $\T({\rm tr}(\ovl{n},\corn{0=0}))\Ra \T({\rm tr}(\ovl{n},\corn{0=0}))$ is available for any $n$ in $\lptn$, and this can be iterated into the transfinite in $\lptinf$.  Moreover, this is achieved without assigning any indices to truth predicates: hypothetical reasoning on truth is automatically grounded in non-truth-theoretic facts, even in the presence of a fully transparent truth predicate. On the other hand, it's also clear that blind hypothetical reasoning, given the undefinability of groudnedness, is only available for non-truth-theoretic sentences. It seems interesting to explore further the connections between $\lpc$, grounded truth, and the associated notion of grounded inference stemming from \cite{niro18}. 

On the logical side, a natural development consists in considering extensions of $\lptn$ with induction principles and more complex truth rules such as general compositional rules. In particular, since the presence of induction prevents full cut elimination arguments, the main focus would be on variants of Proposition \ref{prop:conse} (conservativity property) for such extensions. The main strategy needs to be modified to resemble more closely the conservativity proof-strategy followed in \cite{hal99,lei15} for the compositional, Tarskian truth theory known as $\ctr$, in which one does not require the strong invertibility properties proper of the ${\bf G3}$ systems above. The restriction of initial sequents in that context looks promising because the the counterexamples found to the general strategy in \cite{hal99} -- cf.~\cite[\S3.7]{lei15} -- involve an essential use of contraction and initial sequents involving truth ascriptions.

{\linespread{1} {
\bibliography{mybib}
\bibliographystyle{alpha}}}

\end{document}